\documentclass[10pt,reqno]{amsart}
\usepackage{amsmath,amsthm,amssymb,amsfonts,amscd}
\usepackage{mathrsfs}
\usepackage{bbding}
\usepackage{graphicx,latexsym}
\usepackage{hyperref}
\usepackage{bm}

\newcommand{\bea}{\begin{eqnarray}}
\newcommand{\eea}{\end{eqnarray}}
\newcommand{\bna}{\begin{eqnarray*}}
\newcommand{\ena}{\end{eqnarray*}}

\numberwithin{equation}{section} 

\theoremstyle{plain}
\newtheorem{theorem}{Theorem}[section]
\newtheorem{lemma}{Lemma}[section]

\newtheorem{proposition}{Proposition}[section]

\theoremstyle{definition}

\newtheorem{remark}{Remark}

\begin{document}

\title{Higher order divisor functions over values of mixed powers}

\author{Chenhao Du and Qingfeng Sun }

\date{\today}

\begin{abstract}
Let $\tau_k(n)$ be the $k$-th divisor function. In this paper,
we derive an asymptotic formula
for the sum
$$
\sum_{1\leq n_1,n_2, \dots, n_{\ell}\leq X^{\frac{1}{r}}
\atop 1\leq n_{\ell+1}\le X^{\frac{1}{s}}}\tau_k(n_1^r+n_2^r+\dots +n_{\ell}^r+n_{\ell+1}^s),
$$
where $k\geq 4$, $r\geq 2$, $s\geq 2$
and $\ell\geq 2$ are integers.
Previously only special cases are studied.
\end{abstract}

\thanks{Q. Sun was partially
  supported by the National Natural Science Foundation
  of China (Grant Nos. 11871306 and 12031008) and the Natural
  Science Foundation of Shandong Province (Grant No. ZR2023MA003)}

\maketitle

\section{Introduction}

Let $\tau_k(n)$ be the $k$-th divisor function
which counts the number of representations of $n$ as the product of
$k$ positive integers,
i.e.,
$$
\tau_k(n)=\left\{(n_1,\ldots,n_k)\in \mathbb{N}^k|n_1\cdots n_k=n\right\},
$$
so its Dirichlet series is $\zeta^{k}(z)$.
When $k=2$, $\tau_2(n)=\tau(n)$ is the well-known divisor function, which is also
one of the most prominent arithmetic functions.
If $n=p_1^{\alpha_1}\cdots p_t^{\alpha_t}$, then
$$\tau_k(n)=\left(\begin{array}{c}
                    \alpha_1+k-1 \\
                    k-1
                  \end{array}\right)\cdots\left(\begin{array}{c}
                    \alpha_t+k-1 \\
                    k-1
                  \end{array}\right).$$
Thus the values of a single $\tau_k(n)$ can fluctuate as $n$ varies.
Although individual values does not bring us much interesting information,
the average behaviour of the divisor function is reasonably stable and attractive,
as the charm demonstrated by the celebrated theorem of Dirichlet
$$
\sum_{n\leq X}\tau(n)=X\log X+(2\gamma-1)X+O(\sqrt{X}),
$$
where $\gamma$ is Euler's constant.
Now there is even stunningly beautiful asymptotic formula for the
Titchmarsh divisor problem (see Titchmarsh \cite{T}, Fouvry \cite{F} and
Bombieri, Friedlander and Iwaniec \cite{BFI})
$$
\sum_{p\leq X\atop p \,\rm{prime}}\tau(p+1)= \frac{\zeta(2) \zeta(3)}{
 \zeta(6)} X+c \cdot \mathrm{Li}(X)+O(X(\log X)^{-A})
$$
for some constant $c$ and any $A>0$, where $\mathrm{Li}(X)$ denotes the logarithmic integral.
In order to explore the statistical laws of the divisor functions
at various levels deeply, the mean values of $\tau_k(n)$
over different sparse sequences haves been studied by many authors.
To describe the historical results, let
$F({\bf x})\in \mathbb{Z}[x_1,\ldots,x_{\ell}]$ be an $\ell$-variable nonsingular polynomial.
Consider the sum
$$
\sum_{{\bf x}\in X\mathcal{B}\bigcap \mathbb{N}^{\ell}}\tau_k(F({\bf x}))
$$
for $X\rightarrow \infty$, where $\mathcal{B}\subset \mathbb{R}^{\ell}$
is an $\ell$-dimensional box such that
$\min_{x\in \in X\mathcal{B}} F({\bf x}) \geq 0$ for all sufficiently large $X$.
We would like to recall the following special cases
that have been studied in the literature:
\begin{itemize}
  \item  $k=2$, $F({\bf x})=x_1^2+a$, $-a$ is not a perfect square number: Hooley \cite{Hooley};

\medskip

  \item  $k=2$, $F({\bf x})=x_1^2+x_2^2$: Gafurov \cite{G}, Yu \cite{Yu};

\medskip

  \item $k=2$, $F({\bf x})=x_1^2+x_2^2+x_3^3$: Calder\'{o}n and Velasco \cite{CV},
  Guo and Zhai \cite{GZ}, Zhao \cite{Z};

\medskip

  \item $k=2$, $r\geq 3$, $F({\bf x})=x_1^2+x_2^2+x_3^r$: L\"{u} and Mu \cite{LM};

  \medskip

  \item $k=2$, $r\geq 2$, $\ell\geq 3$, $F({\bf x})=x_1^r+\cdots+x_{\ell}^r$:
Zhang \cite{Zangmeng} ;

 \medskip

  \item  $k=3$, $F({\bf x})=x_1^2+x_2^6$, $(x_1,x_2)=1$: Friedlander and Iwaniec \cite{FI};

\medskip

  \item $k=3$, $F({\bf x})=x_1^2+x_2^2+x_3^2$: Sun and Zhang \cite{SZ};

\medskip

  \item $k=3$, $r\geq 3$, $F({\bf x})=x_1^2+x_2^2+x_3^r$: Zhou and Hu \cite{ZH};

\medskip

  \item $k=3$, $\ell\geq 3$, $F({\bf x})=x_1^2+\cdots+x_{\ell}^2$:
  Hu and Hu \cite{HH} ;

\medskip

  \item $k\geq 4$, $\ell\geq 3$, $F({\bf x})=x_1^2+\cdots+x_{\ell}^2$:
  Hu and L{\"u} \cite{HL};

\medskip

  \item $k\geq 4$, $r\geq 2$, $\ell>2^{r-1}$, $F({\bf x})=x_1^r+\cdots+x_{\ell}^r$:  Zhou and Ding \cite{ZD};
\end{itemize}
More cases and relevant developments can be found in Daniel \cite{D}, Tipu \cite{Tipu},
Blomer \cite{B1},\cite{B2},
Bret\`{e}che and Browning \cite{Br1},
Browning \cite{Br2},
Hu and Yang \cite{Hu3}, Liu \cite{Liu},
Lapkova and Zhou \cite{LZ}, Zhao, Zhai and Li \cite{ZZL} and the references therein.

In this paper, we are concerned with the following sum of higher order
divisor function over values of mixed powers
  \bna
\sum_{1\leq n_1,n_2, \dots ,n_{\ell}
\leq X^{\frac{1}{r}}\atop 1\leq n_{\ell+1}\le X^{\frac{1}{2}}}
\tau_k(n_1^r+n_2^r+\dots +n_{\ell}^r+n_{\ell+1}^s),
\ena
where $k \geq 4$, $r \geq 2$, $s\geq 2$ and $\ell\geq 2^{r-1}$.

To state our results, we introduce some notations.
Let $G_r(a,b;q)$ be the Gauss sum
\bea\label{Gauss sum}
G_r(a,b;q)=\sum_{x\bmod q}e\Big(\frac{ax^r+bx}{q}\Big).
\eea
For $0\le j\le k-1$, we define
\bea\label{Aj1}
A_j(q)=\sum_{b=1}^{q}e\left(-\frac{ab}{q}\right)c_{j+1}(b,q),
\eea
where $a>0$ is an integer. The coefficients $c_j(b,q)$ are sums of terms of the form
\bna
\sum_{b_1b_2 \equiv b\bmod q }f(b_1)
\ena
for some function $f$. The number of terms in $c_{j}(b,q)$ depends only on $k$.
The accurate definition of $A_j(q)$ is a bit cumbersome and we do not need it in
the proof. So we omit its definition and recommend the reader to (2.13) in \cite{C1} for the
details. For our purpose, we only note that
(see (4.8) in \cite{C2})
\bea\label{Aj2}
A_j(q)\ll_{k}q^{-1}
\eea
and $A_j(q)$ is independent of $a$.
For  $0 \leq i \leq j \leq k-1$ , we define
\bea\label{S}
\mathfrak{S}_{k, r,s, \ell, j}=\frac{1}{j!} \sum_{\substack{q=1}}^{\infty}
\sum_{1\le a \le q \atop (a,q)=1} \frac{G_r^{\ell}(a,0;q)G_s(a,0;q)}{q^{\ell+1}} A_{j}(q)
\eea
and
\bea\label{J}
\mathfrak{J}_{r,s, \ell, i}=\int_{-\infty}^{\infty}
\left(\int_{0}^{\ell+1} e(-\beta u_1) \log ^{i} u_1 d u_1\right)
\left(\int_{0}^{1} e\left(u_2^{r} \beta\right) \mathrm{d} u_2\right)^{\ell}
\left(\int_{0}^{1} e\left(u_3^{s} \beta\right) \mathrm{d} u_3\right)\mathrm{d} \beta.
\eea
The main result of this paper is the following asymptotic formula.
\begin{theorem}\label{Thm}Let $k \geq 4$, $r \geq 2$,
$s\geq 2$ and $\ell\geq  2^{r-1}$  be integers. There is a constant
$\delta_{k, r,s, \ell}>0$ such that for any $\varepsilon>0$
\bna
&&\sum_{1\leq n_1,n_2, \dots n_{\ell}\leq X^{\frac{1}{r}}\atop 1\leq n_{\ell+1}\le X^{\frac{1}{s}}}\tau_k(n_1^r+n_2^r+\dots +n_{\ell}^r+n_{\ell+1}^s)\nonumber \\
&= & \sum_{j=0}^{k-1} \mathfrak{S}_{k, r,s, \ell, j}
\sum_{i=0}^{j}\binom{j}{i} \mathfrak{J}_{r,s, \ell, i}
X^{\frac{\ell}{r}+\frac{1}{s}}(\log X)^{j-i}
+O\left(X^{\frac{\ell}{r}+\frac{1}{s}-\delta_{k, r,s, \ell}+\varepsilon}\right),
\ena
where $\mathfrak{S}_{k, r,s, \ell, j}$ and
$\mathfrak{J}_{r,s, \ell, i}$ are defined in \eqref{S} and \eqref{J}, respectively.
Here the exact values of $\delta_{k, r,s, \ell}$ can be found in Section \ref{delta}.
\end{theorem}

\begin{remark}

The magnitude of the leading term is $X^{\frac{\ell}{r}+\frac{\ell}{s}}\log^{k-1}X$,
which is consistent with the
known cases. Furthermore, although
we do not attempt to obtain the best value for the positive constant
$\delta_{k, r,s, \ell}$ for the sake of conciseness, the savings in the $O$-terms
reflected in the values of $\delta_{k, r,s, \ell}$ are
consistent with the best existing results
(except for a few isolated cases; see Section \ref{delta}).
So we generalize previous results.
\end{remark}

\medskip

\section{Proof of the main result}
\setcounter{equation}{0}
\medskip
We will prove Theorem \ref{Thm} by the circle method. In order to apply the circle method, we choose the parameters
$P$ and $Q$ such that
\bna
P=X^{\theta}, \qquad Q=X^{1-\theta},
\ena
where $\theta$ is a positive number to be decided later.
By Dirichlet lemma on rational approximation, each
$\alpha\in I:=\left[Q^{-1},1+Q^{-1}\right]$ can be written in the form
\bea\label{rational approximations}
\alpha=\frac{a}{q}+\beta,\qquad |\beta|\le\frac{1}{qQ}
\eea
for some integers $a$, $q$ with $1\le a\le q\le Q$ and $(a,q)=1$.
We denote by $ \mathfrak{M}(a,q)$ the set of $\alpha$ satisfying \eqref{rational approximations} and define the major arcs and the minor arcs as follows:
\bna
\mathfrak{M}=\bigcup_{1\le q\le P}\bigcup_{1\le a\le q\atop (a,q)=1}
\mathfrak{M}(a,q),\qquad
\mathfrak{m}=\Big[1/Q,1+1/Q\Big] \backslash \mathfrak{M}.
\ena
Let
\bea\label{FT}
F(\alpha,X)=\sum_{1\le n \le (\ell+1)X}\tau_k(n)e(-\alpha n),\qquad T_r\left(\alpha,X\right)=\sum_{1\le n \le X^{\frac{1}{r}} }e(\alpha n^r).
\eea
Then by the orthogonality relation
\begin{center}
	
	$\int_{0}^{1} e(n\alpha)d\alpha=\begin{cases}
	1,&\text{if } n=0,\\
	0,&\text{if } n\in \mathbb{Z}\backslash \{0\},
	
\end{cases}$

\end{center}
we have
\bna
\sum_{1\leq n_1,n_2, \dots, n_{\ell}\leq X^{\frac{1}{r}}
\atop 1\leq n_{\ell+1}\le X^{\frac{1}{s}}}
\tau_k(n_1^r+n_2^r+\dots +n_{\ell}^r+n_{\ell+1}^s)=\int_{0}^{1} T_r^{\ell}(\alpha,X)T_s(\alpha,X)F(\alpha,X)\mathrm{d}\alpha.
\ena
Note that $ T_r^\ell(\alpha,X)T_s(\alpha,X)F(\alpha,X)$ is a periodic function of period 1, one further has
\bea\label{zong-fen}
&&\sum_{1\leq n_1,n_2, \dots ,n_{\ell}\leq X^{\frac{1}{r}}\atop
1\leq n_{\ell+1}\le X^{\frac{1}{s}}}\tau_k(n_1^r+n_2^r+\dots +n_{\ell}^r+n_{\ell+1}^s)
\nonumber\\&=&\int_{\frac{1}{Q}}^{1+\frac{1}{Q}}T_r^\ell(\alpha,X)
T_s(\alpha,X)F(\alpha,X)\mathrm{d}\alpha\nonumber\\
&=&{\int_\mathfrak{M}T_r^\ell(\alpha,X)
T_s(\alpha,X)F(\alpha,x)\mathrm{d}\alpha+\int _\mathfrak{m}}
T_r^{\ell}(\alpha,X)T_s(\alpha,X)F(\alpha,X)\mathrm{d}\alpha.
\eea

For the integral over the major arcs $\mathfrak{M}$, we have the following result
which we shall prove in Section 3.
\begin{proposition}\label{pro}
Let $F(\alpha, X)$ and $T_r(\alpha, X)$ be defined as in \eqref{FT}.
There is a constant $\delta_{k, r,s, \ell}>0$ such that for any $\varepsilon>0$,
\bea
&&\int_\mathfrak{M}T_r^\ell(\alpha,X)T_s\left(\alpha,X\right)
F(\alpha,x)d\alpha\nonumber\\&=&\sum_{j=0}^{k-1}
\mathfrak{S}_{k, r,s, \ell, j} \sum_{i=0}^{j}\binom{j}{i}
\mathfrak{J}_{r,s, \ell, i} X^{\frac{\ell}{r}+\frac{1}{s}}
\left(\log X\right)^{j-i}+
O\left(X^{\frac{\ell}{r}+\frac{1}{s}-\delta_{k, r, s, \ell}+\varepsilon}\right),
\eea
where $\mathfrak{S}_{k, r,s, \ell, j}$ and $\mathfrak{J}_{r,s, \ell, i}$
are defined as in Theorem \ref{Thm}.
\end{proposition}
For the integral from the minor arcs, we will use Hua's Lemma (see Lemma 2.5 in \cite{V}).
\begin{lemma}\label{Hua's lemma}
	Suppose that $1\le j \le k$. Then
	\bna
	\int_{0}^{1}\left|\sum_{1\le m\le N}e(\alpha m^k) \right|^{2^j}d\alpha \ll N^{2^j-j+\varepsilon} .
	\ena
\end{lemma}
By Cauchy-Schwarz's inequality one has
\bea\label{minor arcs}
&&\int_\mathfrak{m}T_r^\ell(\alpha,X)T_s(\alpha,X)F(\alpha,X)\mathrm{d}\alpha\nonumber\\&\ll&
\sup_{\alpha \in \mathfrak{m}} \left|T_s(\alpha,X)\right|\left|T_r(\alpha,X)\right|^{\ell-2^{r-1}}\Big(\int_\mathfrak{m}\left| T_r(\alpha,X) \right|^{2^r}\mathrm{d}\alpha\Big)^\frac{1}{2}\Big(\int_\mathfrak{m} \left|F(\alpha,X )\right|^2\mathrm{d}\alpha \Big)^\frac{1}{2}.
\eea
For the first integral in \eqref{minor arcs}, we apply Lemma \ref{Hua's lemma} to get
\bna
\int_\mathfrak{m} |T_r\left(\alpha, X\right)|^{2^{r}} \mathrm{d} \alpha\le \int_{0}^{1} |T_r\left(\alpha, X\right)|^{2^{r}} \mathrm{d} \alpha \ll X^{\frac{2^{r}}{r}-1+\varepsilon}.
\ena
For the last integral in \eqref{minor arcs}, a result of Shiu (Theorem 2 in \cite{Shiu}) shows that
\bna
&&\int_\mathfrak{m} \left|F(\alpha, X)\right|^{2} \mathrm{d} \alpha \le
\int_{0}^{1} \left|F(\alpha, X)\right|^{2} \mathrm{d} \alpha
=\int_{0}^{1}\Big|\sum_{1\le n \le (\ell+1)X}\tau_{k}(n)e(-\alpha n)\Big|^{2}\mathrm{d}\alpha
\\ &=& \sum_{1\le n \le (\ell+1)X}\tau_{k}^{2}(n) \ll X^{1+\varepsilon}.
\ena
Plugging these estimates into \eqref{minor arcs} we abtain
\bea\label{minor arcs 1}
\int_\mathfrak{m}T_r^\ell(\alpha,X)T_s(\alpha,X)F(\alpha,X)\mathrm{d}\alpha \ll
X^{\frac{2^{r-1}}{r}+\varepsilon}\sup_{\alpha \in \mathfrak{m}} |T_s(\alpha,X)||T_r(\alpha,X)|^{\ell-2^{r-1}}.
\eea
In order to make the error term in the asymptotic formula as small as possible, we distinguish four cases according to the values of $r$ and $s$.\\

 \quad  (\romannumeral1) $2\le r\le 7$, $2\le s\le 7$.

  In this case we apply  Weyl's inequality (see Lemma 2.4 in \cite{V}) to get, for $\alpha \in \mathfrak{m}$
 \bea\label{T_r} T_r\left(\alpha, X\right) \ll X^{\frac{1}{r}+\varepsilon}(P^{-1}+X^{-\frac{1}{r}}+QX^{-1})^{\frac{1}{2^{r-1}}} \ll X^{\frac{1}{r}-\frac{\theta}{2^{r-1}}+\varepsilon}, \eea
for $\theta\le\frac{1}{k+r}$, and similarly,
 \bea\label{T_s} T_s\left(\alpha, X\right)\ll X^{\frac{1}{s}-\frac{\theta}{2^{s-1}}+\varepsilon}. \eea
 Thus by \eqref{minor arcs 1}-\eqref{T_s}, one has
 \bea\label{romannumeral1}
&&\int_\mathfrak{m}T_r^\ell(\alpha,X)T_s(\alpha,X)F(\alpha,X)\mathrm{d}\alpha\nonumber\\ &\ll&
X^{\frac{2^{r-1}}{r}+\varepsilon}\cdot X^{\frac{1}{s}-\frac{\theta}{2^{s-1}}+\varepsilon}\cdot X^{\left(\frac{1}{r}-\frac{\theta}{2^{r-1}}+\varepsilon\right)\left(\ell-2^{r-1}\right)}
\nonumber\\ &\ll& X^{\frac{\ell}{r}+\frac{1}{s}-\theta\left(\frac{\ell}{2^{r-1}}+\frac{1}{2^{s-1}}-1\right)+\varepsilon}.
\eea

 \quad  (\romannumeral2) $2\le r\le 7$, $s\ge 8$.

 In this case, we use Lemma 1.6 of \cite{LM} in place of Weyl's inequality to get
 \bea\label{T_s1} T_s\left(\alpha, X\right) \ll X^{\frac{1}{s}+\varepsilon}\left(P^{-1}+X^{-\frac{1}{s}}+QX^{-1}\right)^{\frac{1}{2s(s-1)}} \ll X^{\frac{1}{s}-\frac{\theta}{2s(s-1)}+\varepsilon}. \eea
 \\ Thus by \eqref{minor arcs 1},\eqref{T_r} and \eqref{T_s1}, we have
 \bea\label{romannumeral2}
&&\int_\mathfrak{m}T_r^\ell(\alpha,X)T_s(\alpha,X)F(\alpha,X)\mathrm{d}\alpha\nonumber\\ &\ll&
X^{\frac{2^{r-1}}{r}+\varepsilon}\cdot X^{\frac{1}{s}-\frac{\theta}{2s(s-1)}+\varepsilon}\cdot X^{\left(\frac{1}{r}-\frac{\theta}{2^{r-1}}+\varepsilon\right)\left(\ell-2^{r-1}\right)}
\nonumber\\ &\ll& X^{\frac{\ell}{r}+\frac{1}{s}-\theta\left(\frac{\ell}{2^{r-1}}+\frac{1}{2s(s-1)}-1\right)+\varepsilon}.
\eea

\quad  (\romannumeral3) $r\ge8$, $2\le s\le 7$.

Similarly as in the case (\romannumeral2),
 \bea\label{T_r1} T_r\left(\alpha, x\right) \ll X^{\frac{1}{r}-\frac{\theta}{2r(r-1)}+\varepsilon}, \eea
and by \eqref{minor arcs 1}, \eqref{T_s} and \eqref{T_r1},
 \bea\label{romannumeral3}
&&\int_\mathfrak{m}T_r^\ell(\alpha,X)T_s(\alpha,X)F(\alpha,X)\mathrm{d}\alpha\nonumber\\ &\ll&
X^{\frac{2^{r-1}}{r}+\varepsilon}\cdot X^{\frac{1}{s}-\frac{\theta}{2^{s-1}}+\varepsilon}\cdot X^{\left(\frac{1}{r}-\frac{\theta}{2r(r-1)}+\varepsilon\right)\left(\ell-2^{r-1}\right)}
\nonumber\\ &\ll& X^{\frac{\ell}{r}+\frac{1}{s}-\theta\left(\frac{\ell-2^{r-1}}{2r(r-1)}+\frac{1}{2^{s-1}}\right)+\varepsilon}.
\eea

\quad  (\romannumeral4) $r\ge8$, $s\ge 8$.

By \eqref{minor arcs 1}, \eqref{T_s1} and \eqref{T_r1}, one has
 \bea\label{romannumeral4}
&&\int_\mathfrak{m}T_r^\ell(\alpha,X)T_s(\alpha,X)F(\alpha,X)\mathrm{d}\alpha\nonumber\\ &\ll&
X^{\frac{2^{r-1}}{r}+\varepsilon}\cdot X^{\frac{1}{s}-\frac{\theta}{2s(s-1)}+\varepsilon}\cdot X^{\left(\frac{1}{r}-\frac{\theta}{2r(r-1)}+\varepsilon\right)\left(\ell-2^{r-1}\right)}
\nonumber\\ &\ll& X^{\frac{\ell}{r}+\frac{1}{s}-\theta\left(\frac{1}{2s(s-1)}+\frac{\ell-2^{r-1}}{2r(r-1)}\right)+\varepsilon}.
\eea
From Proposition \ref{pro}, \eqref{romannumeral1},
\eqref{romannumeral2}, \eqref{romannumeral3} and \eqref{romannumeral4},
Theorem \ref{Thm} follows.
\section{The integral over the major arcs}
In this section, we will give a proof of Proposition \ref{pro}.
By the definition of major arcs,  we clearly have
\bea\label{major1}
&&\int_{\mathfrak{M}}T_r^\ell(\alpha,X)T_s(\alpha,X)F(\alpha,X)\mathrm{d}\alpha \nonumber\\
&=&\sum_{q \leq P} \sum_{1\le a \le q \atop (a,q)=1}\;\int_{\mathfrak{M}(a,q)}T_r^\ell(\alpha,X)T_s(\alpha,X)F(\alpha,X)\mathrm{d}\alpha \nonumber\\
&=&\sum_{q \leq P} \int_{|\beta| \leq \frac{1}{q Q}}
\sum_{1\le a \le q \atop (a,q)=1} T_r^{\ell}\left(\frac{a}{q}+\beta, X\right)
T_s\left(\frac{a}{q}+\beta, X\right)F\left(\frac{a}{q}+\beta, X\right)  \mathrm{d}\beta.
\eea
Now we are in a position to introduce the estimates concerning
the above two functions  $T_r\left(a/q+\beta, X\right)$
and  $F\left(a/q+\beta, X\right).$ We first quote the
following asymptotic formula  for $F\left(a/q+\beta,X\right)$
(see Section 4 in Chace \cite{C2} or Lemma 3.1 in Hu and L{\"u} \cite{HL}).

\begin{lemma}\label{F}Assume $k\ge 4$. Suppose that $(a,q)=1,$ $q \leq P \leq X^{1 / k}$ and $|\beta| \leq 1 /(q Q)$. We have
\bna
F\left(\frac{a}{q}+\beta, X\right)=\sum_{j=0}^{k-1} A_{j}(q) I_{j}(\beta)+O\left(P^{k+\varepsilon}+X^{\eta+\varepsilon} P\right),
\ena
where  $\eta=(k-1) /(k+2),$ $A_{j}(q)$ are defined as in \eqref{Aj1},
and
\bea\label{Ij1}
I_{j}(\beta)=\int_{1}^{(\ell+1)X}e(-\beta u) \frac{\log^{j}u}{j!}du.
\eea
Moreover, we have
\bea\label{Ij2}
I_{j}(\beta)\ll_k X^{\varepsilon}\min \left \{X,|\beta|^{-1}\right \}.
\eea
\end{lemma}
We also need some estimates for $T_r(a/q+\beta,X)$.
\begin{lemma}\label{T}
Let $(a, q)=1$, and $|\beta| \leq 1/(qQ)$. We have
\bna
T_r\left(\frac{a}{q}+\beta, X\right)=\frac{G_r(a,0;q)} q\Psi_r(\beta)
+O\left(q^{\frac{1}{2}+\varepsilon}\left(1+ |\beta|X\right)^\frac{1}{2}\right),
\ena
where $G_r(a,0;q)$ is defined in \eqref{Gauss sum} and
\bea\label{Psi0}
\Psi_r(\beta)=\int_{0}^{X^{1/r}} e\left(\beta u^{r} \right) \mathrm{d} u  .
\eea
Moreover
\bea\label{Psi}
\Psi_r(\beta)\ll\left(\frac{X}{1+|\beta|X} \right)^{1/r}.
\eea
\end{lemma}
\begin{proof}
  The asymptotic formula can be found in Vaughan \cite{V} (see Theorem 4.1). \eqref{Psi} follows from the $r$-th derivative test together with the trival estimate.
\end{proof}
By Lemma \ref{T}, we have
\bea\label{Tr}
T_r\left(\frac{a}{q}+\beta, X\right)=\frac{G_r(a,0;q)} q\Psi_r(\beta)+E_r(q,\beta),
\eea
where
\bna
E_r(q,\beta)\ll q^{\frac{1}{2}+\varepsilon}\left(1+|\beta|X\right)^\frac{1}{2}.
\ena
Thus
\bna
T_r^\ell\left(\frac{a}{q}+\beta, X\right)=\sum_{i=0}^{\ell}\binom{\ell}{i}
\frac{G_r^{\ell-i}(a,0;q)} {q^{\ell-i}}\Psi_r^{\ell-i}(\beta)E_r(q,\beta)^i
\ena
and
\bna
&&\sum_{1\le a \le q \atop (a,q)=1}T_r^\ell\left(\frac{a}{q}+\beta, X\right)T_s\left(\frac{a}{q}+\beta, X\right)F\left(\frac{a}{q}+\beta, X\right)\nonumber\\&=&
\sum_{i=0}^{\ell}\binom{\ell}{i}\mathbf{R}_{i}(q,\beta)
T_s\left(\frac{a}{q}+\beta,X\right)+
O\Bigg(\left(P^{k+\varepsilon}+X^{\eta+\varepsilon}P\right)
\sum_{1\le a \le q \atop (a,q)=1}
\Big| T_r\left(\frac{a}{q}+\beta, X\right)\Big|^{\ell}
\Big| T_s\left(\frac{a}{q}+\beta, X\right)\Big| \Bigg),
\ena
where
\bea\label{Ri}
\mathbf{R}_i(q,\beta)=\frac{\Psi_r^{\ell-i}(\beta)E_r(q,\beta)^{i}}{q^{\ell-i}} \sum_{j=0}^{k-1} I_{j}(\beta) \sum_{1\le a \le q \atop (a,q)=1}  A_{j}(q)G_r^{\ell-i}(a,0;q).
\eea
Inserting this into \eqref{major1}, we get
\bea\label{med-estimate}
&& \int_{\mathfrak{M}} T_{r}^{\ell}(\alpha, X) T_{s}(\alpha, X) F(\alpha, X)  \mathrm{d} \alpha \nonumber\\
&=& \sum_{i=0}^{\ell}\binom{\ell}{i} \sum_{q \leqslant P}
\int_{|\beta| \leqslant \frac{1}{q Q}}\mathbf{R}_i(q, \beta)
T_{s}\left(\frac{a}{q}+\beta, X\right)  \mathrm{d} \beta \nonumber\\
&+&O\left(\left(P^{k+\varepsilon}+X^{\eta+\varepsilon}P\right)
\int_{\mathfrak{M}}\left|T_{r}(\alpha, X)\right|^{\ell}\left|T_{s}(\alpha, X)\right|
 \mathrm{d} \alpha\right) .
\eea
\subsection{Estimate of the $O$-term}

We distinguish two cases. Assume
\bea\label{theta}
\theta\leq \frac{1}{k+r}.
\eea

If $2^{r-1}\leq \ell<2^{r}$ except for $\ell=2^{r}-1$ $(r=s)$, by H\"{o}lder's inequality
and Lemma \ref{Hua's lemma}, the $O$-term is bounded by
\bea\label{O1}
&&\left(P^{k+\varepsilon}+X^{\eta+\varepsilon} P\right)
\int_{\mathfrak{M}}\left|T_r\left(\alpha, X\right)\right|^{\ell}
\left|T_s\left(\alpha, X\right)\right| \mathrm{d} \alpha \nonumber\\
& \ll&\left(P^{k+\varepsilon}+X^{\eta+\varepsilon} P\right)
\sup_{\alpha \in I} \left|T_s\left(\alpha, X\right)\right|^{\frac{\ell}{2^{r-1}}-1}
\left(\int_{0}^{1}\left|T_r\left(\alpha, X\right)\right|^{2^{r}}
\mathrm{d} \alpha\right)^{\frac{\ell}{2^{r}}}\left(\int_{0}^{1}\left|T_s\left(\alpha, X\right)\right|^2
\mathrm{d} \alpha\right)^{\frac{2^r-\ell}{2^r}}\nonumber\\
& \ll&\left(X^{k\theta+\varepsilon}+
X^{\frac{k-1}{k+2}+\theta+\varepsilon}\right)
X^{\frac{\ell}{r}-\frac{\ell}{2^{r}}(1-\frac{1}{s})+\varepsilon}\nonumber\\
&\ll&X^{\frac{\ell}{r}-\frac{\ell}{2^{r}}(1-\frac{1}{s})+\frac{k-1}{k+2}+\theta+\varepsilon},
\eea
where the last inequality follow from the assumption in \eqref{theta}.

If $\ell \geq 2^{r}$ or $\ell=2^{r}-1$ $(r=s)$, by Lemma \ref{Hua's lemma}, the $O$-term can be bounded by
\bea\label{O2}
&&\left(P^{k+\varepsilon}+X^{\eta+\varepsilon} P\right)
\int_{\mathfrak{M}}\left|T_r\left(\alpha, X\right)\right|^{\ell}
\left|T_s\left(\alpha, X\right)\right| \mathrm{d} \alpha \nonumber\\
&\ll&\left(P^{k+\varepsilon}+X^{\eta+\varepsilon} P\right) \sup _{\alpha \in I}\left|T_r\left(\alpha, X\right)\right|^{\ell-2^{r}} \left|T_s\left(\alpha, X\right)\right|\left(\int_{0}^{1}\left|T_r\left(\alpha, X\right)\right|^{2^{r}} \mathrm{d} \alpha\right)\nonumber \\
&\ll& \left(X^{k\theta+\varepsilon}+
X^{\frac{k-1}{k+2}+\theta+\varepsilon}\right)
X^{\frac{\ell}{r}-(1-\frac{1}{s})+\varepsilon}\nonumber\\
&\ll&X^{\frac{\ell}{r}-(1-\frac{1}{s})+\frac{k-1}{k+2}+\theta+\varepsilon},
\eea
where the last inequality follow from the assumption in \eqref{theta}.
\subsection{Contribution from $\mathbf{R}_0(q,\beta)$}
By the definition of  $\mathbf{R}_0(q,\beta)$ in \eqref{Ri} and Lemma \ref{T}, we have
\bea\label{R0-1}
&&\sum_{q \leqslant P} \int_{|\beta| \leqslant \frac{1}{q Q}}
\mathbf{R}_{0}(q, \beta) T_{s}\Big(\frac{a}{q}+\beta, X\Big)  \mathrm{d} \beta \nonumber\\
&=& \sum_{q \leqslant P} \int_{|\beta| \leqslant \frac{1}{q Q}}
\mathbf{R}_{0}(q, \beta) \frac{G_{s}(a, 0;q)}{q} \Psi_{s}(\beta)  \mathrm{d} \beta \nonumber\\
&&\qquad+O\Bigg(\sum_{q \leqslant P} q^{\frac{1}{2}+\varepsilon}
\int_{|\beta| \leqslant \frac{1}{q Q}} |\mathbf{R}_{0}(q, \beta)|
 (1+|\beta|X)^{\frac{1}{2}}  \mathrm{d} \beta\Bigg)\nonumber\\
&:=&\Delta_1+\Delta_2,
\eea
where
\bna
\Delta_1=\sum_{j=0}^{k=1}\sum_{q\le P}\sum_{1\le a \le q \atop (a,q)=1}
\frac{ A_{j}(q)G_r^{\ell}(a,0;q)G_s(a,0;q)}{q^{\ell+1}}
\int_{|\beta| \leqslant \frac{1}{q Q}}I_{j}(\beta)\Psi_r^{\ell}(\beta)
\Psi_s(\beta)\mathrm{d}\beta
\ena
and
\bna
\Delta_2\ll \sum_{j=0}^{k=1}\sum_{q\le P}\sum_{1\le a \le q \atop \left(a,q\right)=1}
\frac{ |A_{j}(q)||G_r(a,0;q)|^{\ell}}{q^{\ell-\frac{1}{2}-\varepsilon}}
\int_{|\beta| \leqslant \frac{1}{q Q}}|I_j(\beta)||\Psi_r(\beta)|^\ell
\left(1+|\beta|X\right)^{\frac{1}{2}} \mathrm{d} \beta.
\ena

For $\Delta_{2}$ , by \eqref{Aj2}, \eqref{Ij2}, \eqref{Psi} and \eqref{Gr} we have
\bna
\Delta_{2}
 &\ll&X^{\frac{\ell}{r}}\sum_{q \leqslant P} q^{\frac{1}{2}-\frac{\ell}{r}+\varepsilon} \int_{ |\beta|\le\frac{1}{qQ}}\min\{X,|\beta|^{-1}\} \left(1+X|\beta|\right)^{\frac{1}{2}-\frac{\ell}{r}} \mathrm{d} \beta\nonumber\\
 &\ll&X^{\frac{\ell}{r}} \sum_{q \leqslant P} q^{\frac{1}{2}-\frac{\ell}{r}+\varepsilon} \int_{ |\beta|\le\frac{1}{X}}X \mathrm{d}\beta \nonumber\\
 &&+ X^{\frac{\ell}{r}}\sum_{q \leqslant P} q^{\frac{1}{2}-\frac{\ell}{r}+\varepsilon}
 \int_{\frac{1}{X}< |\beta|\le\frac{1}{qQ}}|\beta|^{-1}
 \left(1+X|\beta|\right)^{\frac{1}{2}-\frac{\ell}{r}} \mathrm{d} \beta.
\ena
Note that the condition $\ell\ge 2^{r-1}(r\ge 2)$ implies $\frac{\ell}{r}\ge 1$. Then
\bea\label{Delta2}
\Delta_2\ll
\begin{cases}
X^{\frac{\ell}{r}+\left(\frac{3}{2}-\frac{\ell}{r}\right)\theta+\varepsilon  },
& \text{ if } \frac{\ell}{r}<\frac{3}{2}, \\
X^{\frac{\ell}{r}+\varepsilon},  & \text{ if } \frac{\ell}{r}\ge\frac{3}{2}
\end{cases}.
\eea

Next we evaluate $\Delta_1$ which contributes the main term.
By Theorem 4.2 in \cite{V}, we have
\bea\label{Gr}
G_r(a,0;q) \ll q^{1-\frac{1}{r}+\varepsilon}.
\eea
This together with \eqref{Aj2}, \eqref{Ij2} and \eqref{Psi} yields that the error term generated by extending the interval over $\beta$ to $(-\infty, \infty)$ does not exceed
\bna
&&\sum_{j=0}^{k-1}\sum_{q \leq P}  \sum_{1\le a \le q \atop (a,q)=1} \frac{ |A_{j}(q)| |G_r^{\ell}(a,0;q)||G_s(a,0;q)|}{q^{\ell+1}} \int_{|\beta|>\frac{1}{q Q}} |I_{j}(\beta)||\Psi_r(\beta)|^{\ell}|\Psi_s(\beta)|  \mathrm{d} \beta \nonumber\\
&\ll&\sum_{q\le P}\frac{1}{q^{\ell+2}}\cdot q^{\left(1-\frac{1}{r}\right)\ell}\cdot q^{1-\frac{1}{s}}\cdot q \int_{|\beta|>\frac{1}{q Q}}|\beta|^{-1-\frac{\ell}{r}-\frac{1}{s}} \mathrm{d} \beta\nonumber\\
&\ll&\sum_{q\le P}q^{-\frac{\ell}{r}-\frac{1}{s}}\cdot (qQ)^{\frac{\ell}{r}+\frac{1}{s}}\nonumber\\
&\ll&PQ^{\frac{\ell}{r}+\frac{1}{s}}\\
&\ll&X^{\frac{\ell}{r}+\frac{1}{s}-(\frac{\ell}{r}+\frac{1}{s}-1)\theta+\varepsilon}.
\ena
Therefore,
\bea\label{Delta1}
\Delta_1=\sum_{j=0}^{k=1}\sum_{q\le P}\sum_{1\le a \le q \atop (a,q)=1}
\frac{ A_{j}(q)G_r^{\ell}(a,0;q)G_s(a,0;q)}{q^{\ell+1}}
\int_{-\infty}^{\infty}I_{j}(\beta)\Psi_r^{\ell}(\beta)\Psi_s(\beta)  \mathrm{d} \beta
+O\big(X^{\frac{\ell}{r}+\frac{1}{s}-
\left(\frac{\ell}{r}+\frac{1}{s}-1\right)\theta+\varepsilon}\big).
\eea
Inserting the definitions of $I_j(\beta)$ and $\Psi_r(\beta)$ in \eqref{Ij1} and \eqref{Psi0} into the integral on the right hand side of \eqref{Delta1}, one has
\bna
\begin{aligned}
& \int_{-\infty}^{\infty}I_{j}(\beta)  \Psi_{r}^{\ell}(\beta) \Psi_{s}(\beta)
\mathrm{d} \beta \\
= & \int_{-\infty}^{\infty}\left(\int_{1}^{(\ell+1) X} e(-\beta u_1)
\frac{\log ^{j} u_1}{j!} \mathrm{d} u_1\right)\left(\int_{0}^{X^{\frac{1}{r}}}
e\left(\beta u_2^{r}\right) \mathrm{d} {u_2}\right)^{\ell}
\left(\int_{0}^{X^{\frac{1}{s}}}
e\left(\beta u_3^{s}\right) \mathrm{d} u_3\right)  \mathrm{d} \beta \\
= & \frac{X^{\frac{\ell}{r}+\frac{1}{s}+1}}{j!}
\int_{-\infty}^{\infty}\left(\int_{\frac{1}{X}}^{\ell+1} e(-\beta Xu_1) \log ^{j} (Xu_1)
\mathrm{d} u_1\right)\left(\int_{0}^{1} e\left(\beta X u_2^{r}\right)
\mathrm{d} u_{2}\right)^{\ell}\left(\int_{0}^{1} e\left(\beta X u_3^{s}\right) \mathrm{d} u_3\right)
 \mathrm{d} \beta \\
= & \frac{X^{\frac{\ell}{r}+\frac{1}{s}}}{j!} \int_{-\infty}^{\infty}
\left(\int_{\frac{1}{X}}^{\ell+1} e(-\beta u_1)\log^{j} (X u_1) \mathrm{d} u_1\right)
\left(\int_{0}^{1} e\left(\beta u_2^{r}\right) \mathrm{d} u_2\right)^{\ell}
\left(\int_{0}^{1} e\left(\beta u_3^{s}\right) \mathrm{d} u_3\right)  \mathrm{d} \beta \\
= & \frac{X^{\frac{\ell}{r}+\frac{1}{s}}}{j!} \int_{-\infty}^{\infty}
\left(\int_{0}^{\ell+1} e(-\beta u_1)\log^{j} (X u_1) \mathrm{d} u_1\right)
\left(\int_{0}^{1} e\left(\beta u_2^{r}\right) \mathrm{d} u_2\right)^{\ell}
\left(\int_{0}^{1} e\left(\beta u_3^{s}\right) \mathrm{d} u_3\right)  \mathrm{d} \beta \\
&+O\left(X^{\frac{\ell}{r}+\frac{1}{s}-1+\varepsilon}\right).
\end{aligned}
\ena
where we have used the estimate
\bna
&&\int_{-\infty}^{\infty}\left(\int_{0}^{\frac{1}{X}} e(-\beta u_1)\log^{j} (X u_1) d u_1\right)\left(\int_{0}^{1} e\left(\beta u_2^{r}\right) d u_2\right)^{\ell}\left(\int_{0}^{1} e\left(\beta u_3^{s}\right) d u_3\right)  \mathrm{d} \beta\nonumber \\
&\ll&\int_{0}^{\frac{1}{X}}\left|\log^j(Xu_1)\right|\int_{-\infty}^{\infty}|\beta|^{-\frac{\ell}{r}-\frac{1}{s}} \mathrm{d} \beta\\
&\ll & X^{-1+\varepsilon}
\ena
for any $\varepsilon>0$. Here we have applied the $r$-th derivative test.
Furthermore, by splitting $\log (Xu)$ we write
\bea\label{integral0}
 \int_{-\infty}^{\infty}I_{j}(\beta)  \Psi_{r}^{\ell}(\beta) \Psi_{s}(\beta)
\mathrm{d} \beta
=  \frac{1}{j!} \sum_{i=0}^{j}\binom{j}{i}
\mathfrak{J}_{r,s,\ell, i}
X^{\frac{\ell}{r}+\frac{1}{s}}(\log X)^{j-i}
+O\left(X^{\frac{\ell}{r}+\frac{1}{s}-1+\varepsilon}\right),
\eea
where $\mathfrak{J}_{r, s, \ell, i}$ is defined in \eqref{J}.
Note that by \eqref{Aj2} and \eqref{Gr}, we have
\bna
\sum_{j=0}^{k-1} \sum_{q \le P }  \sum_{1\le a \le q \atop (a,q)=1}
\frac{|A_{j}(q)| |G_r^{\ell}(a, 0; q)| |G_{s}(a, 0; q)|}{q^{\ell+1}} \ll 1 .
\ena
Thus, by inserting \eqref{integral0} into \eqref{Delta1} one has
\bna
\Delta_1&=&\sum_{j=0}^{k-1} \frac{1}{j!} \sum_{q \leq P}
\sum_{1\le a \le q \atop (a,q)=1}\frac{G_r^{\ell}(a,0;q)G_s(a,0;q)}{q^{\ell+1}}
A_{j}(q) \sum_{i=0}^{j}\binom{j}{i} \mathfrak{J}_{r,s, \ell, i}
X^{\frac{\ell}{r}+\frac{1}{s}}(\log X)^{j-i} \\
&&+O\left(X^{\frac{\ell}{r}+\frac{1}{s}-(\frac{\ell}{r}+\frac{1}{s}-1)\theta+\varepsilon}
+X^{\frac{\ell}{r}+\frac{1}{s}-1+\varepsilon}\right).
\ena
Moreover, by extending the summation over $q$
to all positive integers with an error term at most
\bna
&&\sum_{j=0}^{k-1} \frac{1}{j!} \sum_{q>P}
\sum_{1\le a \le q \atop (a,q)=1} \frac{|G_r^{\ell}(a,0;q)G_s(a,0;q)|}{q^{\ell+1}}
 |A_{j}(q)|
\sum_{i=0}^{j}\binom{j}{i} |\mathfrak{J}_{r,s, \ell, i}|
X^{\frac{\ell}{r}+\frac{1}{s}}(\log X)^{j-i} \nonumber\\
&\ll&X^{\frac{\ell}{r}+\frac{1}{s}+\varepsilon}
\sum_{q>P}\frac{q^{\ell-\frac{\ell}{r}}\cdot q^{1-\frac{1}{s}}\cdot q}{q^{\ell+1}\cdot q}\nonumber\\
&\ll&X^{\frac{\ell}{r}+\frac{1}{s}+\varepsilon}P^{-\frac{\ell}{r}-\frac{1}{s}+1}\\
&\ll&X^{\frac{\ell}{r}+\frac{1}{s}-(\frac{\ell}{r}+\frac{1}{s}-1)\theta+\varepsilon},
\ena
we derive
\bea\label{Delta1-sum}
\Delta_1=\sum_{j=0}^{k-1} \mathfrak{S}_{k,r ,s ,\ell, j}
\sum_{i=0}^{j}\binom{j}{i} \mathfrak{J}_{r,s, \ell, i}
X^{\frac{\ell}{r}+\frac{1}{s}}(\log X)^{j-i}
+O\left(X^{\frac{\ell}{r}+\frac{1}{s}-\left(\frac{\ell}{r}+\frac{1}{s}-1\right)\theta+\varepsilon}
+X^{\frac{\ell}{r}+\frac{1}{s}-1+\varepsilon}\right),
\eea
where $\mathfrak{S}_{k,r ,s ,\ell, j}$ is defined in \eqref{S}.
This combined with\eqref{R0-1} and \eqref{Delta2} yields
\bea\label{R0}
&&\sum_{q \leqslant P} \int_{|\beta| \leqslant \frac{1}{q Q}}
\mathbf{R}_{0}(q, \beta) T_{s}\left(\frac{a}{q}+\beta, X\right)
\mathrm{d} \beta\nonumber\\
&=&\sum_{j=0}^{k-1}  \mathfrak{S}_{k, r, s, \ell, j} \sum_{i=0}^{j}
\binom{j}{i} \mathfrak{J}_{r, s,\ell, i}
X^{\frac{\ell}{r}+\frac{1}{s}}(\log X)^{j-i}\nonumber\\&&+
O\Big(X^{\frac{\ell}{r}+\varepsilon}+X^{\frac{\ell}{r}+\left(\frac{3}{2}-\frac{\ell}{r}\right)\theta+\varepsilon  }+
+X^{\frac{\ell}{r}+\frac{1}{s}-\left(\frac{\ell}{r}+\frac{1}{s}-1\right)\theta+\varepsilon}\Big).
\eea


\subsection{Contribution from $\mathbf{R}_i(q,\beta)$ with $1\le i \le \ell $}

By the definition of $\mathbf{R}_{i}(q, \beta)$ in \eqref{Ri} and \eqref{Tr}, we have
\bna
&&\sum_{q \leqslant P} \int_{|\beta| \leqslant \frac{1}{q Q}}
\mathbf{R}_{i}(q, \beta) T_{s}\left(\frac{a}{q}+\beta, X\right)  \mathrm{d} \beta \nonumber\\
&= &\sum_{j=0}^{k-1} \sum_{q \leqslant P} \sum_{1\le a \le q \atop (a,q)=1}
\frac{A_{j}(q) G_r^{\ell-i}(a, 0;q) G_{s}(a, 0; q)}{q^{\ell+1-i}}
\int_{|\beta|\leq\frac{1}{q Q}} I_{j}(\beta) \Psi_{r}^{\ell-i}(\beta) \Psi_{s}(\beta)E_r(q,\beta)^i \mathrm{d} \beta\\
&& + \sum_{j=0}^{k-1} \sum_{q \le P} \sum_{1\le a \le q \atop (a,q)=1}
 \frac{A_{j}(q) G_r^{\ell-i}(a, 0; q)}{q^{\ell-i}} \int_{|\beta|\le\frac{1}{qQ}} I_{j}(\beta) \Psi_{r}^{\ell-i}(\beta)
E_r(q,\beta)^{i} E_s(q,\beta) \mathrm{d} \beta \\
&:= & \Delta_3+\Delta_4,
\ena
say.

By \eqref{Aj2},  \eqref{Ij2} and \eqref{Gr}, we have
\bna
\Delta_3&\ll&X^{\frac{\ell-i}{r}+\frac{1}{s}}\sum_{q \leqslant P}
q^{-\frac{\ell-i}{r}-\frac{1}{s}+\frac{i}{2}+\varepsilon}
\int_{|\beta|\le\frac{1}{qQ}}\min\{X,|\beta^{-1}|\}
\left(1+|\beta |X \right)^{\frac{i}{2}-\frac{\ell-i}{r}-\frac{1}{s}}
\mathrm{d} \beta
\nonumber\\ &\ll&X^{\frac{\ell-i}{r}+\frac{1}{s}}
\sum_{q \leqslant P}q^{-\frac{\ell-i}{r}-\frac{1}{s}+\frac{i}{2}+\varepsilon}
\int_{|\beta|\le\frac{1}{X}}X \mathrm{d} \beta\nonumber\\
&&+X^{\frac{i}{2}}\sum_{q \leqslant P}
q^{-\frac{\ell-i}{r}-\frac{1}{s}+\frac{i}{2}+\varepsilon}
\int_{\frac{1}{X}< |\beta|\le\frac{1}{qQ}}
|\beta|^{\frac{i}{2}-\frac{\ell-i}{r}-\frac{1}{s}-1}\mathrm{d} \beta\\
&\ll&\begin{cases}
X^{\frac{\ell}{r}+\frac{1}{s}
-\left(\frac{\ell}{r}+\frac{1}{s}-1\right)\theta
-\left(\frac{1}{r}-(\frac{1}{2}+\frac{1}{r})\theta\right)i+\varepsilon},
& \text{ if } \frac{\ell}{r}< 1+i\left( \frac{1}{2}+\frac{1}{r}\right)-\frac{1}{s}, \\
\\
X^{\frac{\ell-i}{r}+\frac{1}{s}+\varepsilon} ,
& \text{ if } \frac{\ell}{r}\ge 1+i\left( \frac{1}{2}+\frac{1}{r}\right)-\frac{1}{s},
\end{cases}
\ena
and similarly,
\bna
\Delta_4&\ll&X^{\frac{\ell-i}{r}}\sum_{q \leqslant P}q^{-\frac{\ell-i}{r}+\frac{i+1}{2}+\varepsilon}\int_{|\beta|\le\frac{1}{qQ}}\min\{X,|\beta^{-1}|\}\left(1+|\beta |X \right)^{\frac{i+1}{2}-\frac{\ell-i}{r}} \mathrm{d} \beta
\nonumber\\ &\ll&X^{\frac{\ell-i}{r}}\sum_{q \leqslant P}q^{-\frac{\ell-i}{r}+\frac{i+1}{2}+\varepsilon}\int_{|\beta|\le\frac{1}{X}}X \mathrm{d} \beta\nonumber\\
&&+X^{\frac{i+1}{2}}\sum_{q \leqslant P}q^{-\frac{\ell-i}{r}+\frac{i+1}{2}+\varepsilon}
\int_{\frac{1}{X}< |\beta|\le\frac{1}{qQ}}
|\beta|^{\frac{i+1}{2}-\frac{\ell-i}{r}-1}\mathrm{d} \beta\\
&\ll&\begin{cases}
X^{\frac{\ell}{r}
+\left(\frac{3}{2}-\frac{\ell}{r}\right)\theta
-\left(\frac{1}{r}-(\frac{1}{2}+\frac{1}{r})\theta\right)i+\varepsilon},
& \text{ if } \frac{\ell}{r}< \frac{3}{2}+i\left( \frac{1}{2}+\frac{1}{r}\right),  \\
\\
X^{\frac{\ell-i}{r}+\varepsilon} ,
& \text{ if } \frac{\ell}{r}\ge \frac{3}{2}+i\left( \frac{1}{2}+\frac{1}{r}\right).
\end{cases}
\ena
Recall that $i\geq 1$, $\frac{1}{2}+\frac{1}{r}\leq 1$ and $\theta<\frac{1}{r}$.
Assembling the above estimates, we get
\bea\label{Ri-final}
&&\sum_{i=1}^{\ell}\binom{\ell}{i}
\sum_{q \leqslant P} \int_{|\beta| \leqslant \frac{1}{q Q}}
\mathbf{R}_{i}(q, \beta) T_{s}\left(\frac{a}{q}+\beta, X\right)
\mathrm{d} \beta\nonumber \\
&\ll& X^{\frac{\ell}{r}+\frac{1}{s}
-\left(\frac{\ell}{r}+\frac{1}{s}-1\right)\theta
-\left(\frac{1}{r}-(\frac{1}{2}+\frac{1}{r})\theta\right)+\varepsilon}
+X^{\frac{\ell}{r}
+\left(\frac{3}{2}-\frac{\ell}{r}\right)\theta
-\left(\frac{1}{r}-\left(\frac{1}{2}+\frac{1}{r}\right)\theta\right)+\varepsilon}
+X^{\frac{\ell-1}{r}+\frac{1}{s}+\varepsilon} .
\eea
\subsection{Conclusion}
By \eqref{med-estimate}, \eqref{O1}, \eqref{O2}, \eqref{R0} and \eqref{Ri-final}, we have
\bea\label{final-estimate}
&& \int_{\mathfrak{M}} T_{r}^{\ell}(\alpha, X) T_{s}(\alpha, X) F(\alpha, X)  \mathrm{d} \alpha \nonumber\\
&=& \sum_{j=0}^{k-1}  \mathfrak{S}_{k, r, s, \ell, j} \sum_{i=0}^{j}
\binom{j}{i} \mathfrak{J}_{r, s,\ell, i}
X^{\frac{\ell}{r}+\frac{1}{s}}(\log X)^{j-i}\nonumber\\&&+
O\Big(X^{\frac{\ell}{r}+\varepsilon}+X^{\frac{\ell}{r}
+\left(\frac{3}{2}-\frac{\ell}{r}\right)\theta+\varepsilon  }
+X^{\frac{\ell}{r}+\frac{1}{s}
-\left(\frac{\ell}{r}+\frac{1}{s}-1\right)\theta+\varepsilon}
+X^{\frac{\ell-1}{r}+\frac{1}{s}+\varepsilon} +\mathbf{\Delta}\Big), \nonumber\\
\eea
where
\bea\label{fff}
\mathbf{\Delta}=\begin{cases}
X^{\frac{\ell}{r}-\frac{\ell}{2^{r}}(1-\frac{1}{s})+\frac{k-1}{k+2}+\theta+\varepsilon},
& \text{ if } 2^{r-1}\leq \ell<2^{r},  \\
\\
X^{\frac{\ell}{r}-(1-\frac{1}{s})+\frac{k-1}{k+2}+\theta+\varepsilon},
& \text{ if } \ell \geq 2^{r}.
\end{cases}
\eea
Taking $\theta$ as in Section 4, then Proposition \ref{pro} follows.
\section{The constant $\delta_{k, r,s, \ell}$}\label{delta}

We first write
\bea\label{zong-fen1}
\sum_{1\leq n_1,n_2, \dots ,n_{\ell}\leq X^{\frac{1}{r}}\atop
1\leq n_{\ell+1}\le X^{\frac{1}{s}}}\tau_k(n_1^r+n_2^r+\dots +n_{\ell}^r+n_{\ell+1}^s)
= \mathbf{M}+
O(\mathbf{E}),
\eea
where $\mathbf{E}$ is the sum of error terms
of order less than $X^{\frac{\ell}{r}+\frac{1}{s}}$ and
$$
\mathbf{M}=\sum_{j=0}^{k-1}  \mathfrak{S}_{k, r, s, \ell, j} \sum_{i=0}^{j}
\binom{j}{i} \mathfrak{J}_{r, s,\ell, i}
X^{\frac{\ell}{r}+\frac{1}{s}}(\log X)^{j-i}.
$$

\subsection{The case of $2^{r-1}\le \ell<2^{r}$  except for $\ell=2^{r}-1$ $(r=s)$}
\qquad

By \eqref{final-estimate} and \eqref{fff}, we have
\bea\label{final-estimate1}
&& \int_{\mathfrak{M}} T_{r}^{\ell}(\alpha, X) T_{s}(\alpha, X) F(\alpha, X)  \mathrm{d} \alpha \nonumber\\
&=& \mathbf{M}+
O\Big(X^{\frac{\ell}{r}+\varepsilon}+X^{\frac{\ell}{r}
+\left(\frac{3}{2}-\frac{\ell}{r}\right)\theta+\varepsilon  }
+X^{\frac{\ell}{r}+\frac{1}{s}
-\left(\frac{\ell}{r}+\frac{1}{s}-1\right)\theta+\varepsilon}
+X^{\frac{\ell-1}{r}+\frac{1}{s}+\varepsilon}
+X^{\frac{\ell}{r}-\frac{\ell}{2^{r}}\left(1-\frac{1}{s}\right)+\frac{k-1}{k+2}+\theta+\varepsilon}\Big). \nonumber\\
\eea

Since $\mathbf{E}$ is the sum of error terms
of order less than $X^{\frac{\ell}{r}+\frac{1}{s}}$,
we must have $\frac{\ell}{2^{r}}(1-\frac{1}{s})-\frac{k-1}{k+2}+\frac{1}{s}>\theta$.
Assuming $\theta>0$, we have $\frac{\ell}{2^{r}}(1-\frac{1}{s})-\frac{k-1}{k+2}+\frac{1}{s}>0$,
which means $k<\frac{3s\cdot2^r}{(s-1)(2^r-\ell)}-2$.

(\romannumeral1) $2\le r\le 7$, $2\le s\le 7$.
Recall \eqref{romannumeral1} which we relabel as
\bea\label{romannumeral1*}
\int_\mathfrak{m}T_r^\ell(\alpha,X)T_s(\alpha,X)F(\alpha,X)\mathrm{d}\alpha
\ll X^{\frac{\ell}{r}+\frac{1}{s}-
\left(\frac{\ell}{2^{r-1}}+\frac{1}{2^{s-1}}-1\right)\theta+\varepsilon}.
\eea
Note that $2^{r-1}\geq 2$ for $r\geq 2$. So the upper bound in \eqref{romannumeral1*}
dominates the third term in \eqref{final-estimate1}. By \eqref{zong-fen} and
\eqref{zong-fen1}-\eqref{romannumeral1*}, we have
\bea\label{err1}
\mathbf{E}=
X^{\frac{\ell}{r}+\varepsilon}+X^{\frac{\ell}{r}
+\left(\frac{3}{2}-\frac{\ell}{r}\right)\theta+\varepsilon  }
+X^{\frac{\ell}{r}+\frac{1}{s}-
\left(\frac{\ell}{2^{r-1}}+\frac{1}{2^{s-1}}-1\right)\theta+\varepsilon}
+X^{\frac{\ell-1}{r}+\frac{1}{s}+\varepsilon}
+X^{\frac{\ell}{r}-\frac{\ell}{2^{r}}\left(1-\frac{1}{s}\right)
+\frac{k-1}{k+2}+\theta+\varepsilon}.
\eea
For $\ell/r\ge \frac{3}{2}$, we choose
\bna
\frac{\ell}{r}+\frac{1}{s}-
\left(\frac{\ell}{2^{r-1}}+\frac{1}{2^{s-1}}-1\right)\theta=\frac{\ell}{r}-\frac{\ell}{2^{r}}\left(1-\frac{1}{s}\right)
+\frac{k-1}{k+2}+\theta,
\ena
i.e.,
\bna
\theta=\frac{\big(k+2\big)\big( 2^r+\ell(s-1)\big)-2^r\big(k-1\big)s}{2s\big(k+2\big)(\ell+2^{r-s})}:=\theta_a.
\ena
Take $\theta_i=\frac{1}{k+r}$. Then
\bna
E\ll X^{\frac{\ell}{r}+\frac{1}{s}-\delta_{k, r, s, \ell}+\varepsilon}
\ena
with
\bna
{\delta_{k, r, s, \ell}}=\begin{cases}
\min\left\{\frac{1}{r},\frac{1}{s},\left(\frac{\ell}{2^{r-1}}+\frac{1}{2^{s-1}}-1\right)\theta_a\right\},
& \text{ if } \theta_a\leq \theta_i,  \\
\\
\min\left\{\frac{1}{r},\frac{1}{s},\left(\frac{\ell}{2^{r-1}}+\frac{1}{2^{s-1}}-1\right)\theta_i\right\},
& \text{ if } \theta_a > \theta_i.
\end{cases}
\ena
For $\ell/r<\frac{3}{2}$, we have two choices.

\begin{itemize}
  \item $r=2$, $\ell=2$.
 In this case, $\frac{\ell}{r}=1$, $\frac{\ell}{2^{r-1}}=1$.
\end{itemize}
 We note that
\bna
\frac{1}{s}-\frac{1}{2}\theta_a\ge\frac{1}{2^{s-1}}\theta_a.
\ena
Then
\bna
E\ll X^{\frac{\ell}{r}+\frac{1}{s}-\delta_{k, r, s, \ell}+\varepsilon}
\ena
with
\bna
{\delta_{k, r, s, \ell}}=\begin{cases}
\frac{1}{2^{s-1}}\theta_a,
& \text{ if } \theta_a\leq \theta_i,  \\
\\
\frac{1}{2^{s-1}}\theta_i,
& \text{ if } \theta_a > \theta_i,
\end{cases}
\ena
where
\bna
\theta_a=\frac{(k+2)(s+1)-2(k-1)s}{2s(k+2)(1+2\cdot2^{-s})}.
\ena
Taking $s=2$, we have
\bna
\theta_a=\frac{10-k}{6k+12},\qquad \delta_{k, r, s, \ell}=\frac{10-k}{12k+24}.
\ena
According to our conditions for $k$, this leads directly to the
results for $4\le k \le9$ in Hu and L{\"u} \cite{HL}.

\begin{itemize}\item
$r=3$, $\ell=4$. In this case, $\frac{\ell}{r}=\frac{4}{3}$, $\frac{\ell}{2^{r-1}}=1$.
\end{itemize}
 We note that
\bna
\frac{1}{s}-\frac{1}{6}\theta_a\ge\frac{1}{2^{s-1}}\theta_a.
\ena
Then
\bna
E\ll X^{\frac{\ell}{r}+\frac{1}{s}-\delta_{k, r, s, \ell}+\varepsilon}
\ena
with
\bna
{\delta_{k, r, s, \ell}}=\begin{cases}
\frac{1}{2^{s-1}}\theta_a,
& \text{ if } \theta_a\leq \theta_i,  \\
\\
\frac{1}{2^{s-1}}\theta_i,
& \text{ if } \theta_a > \theta_i,
\end{cases}
\ena
where
\bna
\theta_a=\frac{(k+2)(s+1)-2(k-1)s}{2s(k+2)(1+2\cdot2^{-s})}.
\ena

 (\romannumeral2). $2\le r\le 7$, $s\ge 8$.
Recall \eqref{romannumeral2} which we relabel as
\bea\label{romannumeral2*}
\int_\mathfrak{m}T_r^\ell(\alpha,X)T_s(\alpha,X)F(\alpha,X)\mathrm{d}\alpha
\ll X^{\frac{\ell}{r}+\frac{1}{s}-\theta\left(\frac{\ell}{2^{r-1}}+\frac{1}{2s(s-1)}-1\right)+\varepsilon}.
\eea
Note that $2^{r-1}\geq 2$ for $r\geq 2$. So the upper bound in \eqref{romannumeral2*}
dominates the third term in \eqref{final-estimate1}. By \eqref{zong-fen}
\eqref{zong-fen1}, \eqref{final-estimate1} and \eqref{romannumeral2*}, we have
\bea\label{err2}
\mathbf{E}=
X^{\frac{\ell}{r}+\varepsilon}+X^{\frac{\ell}{r}
+\left(\frac{3}{2}-\frac{\ell}{r}\right)\theta+\varepsilon  }
+X^{\frac{\ell}{r}+\frac{1}{s}-
\left(\frac{\ell}{2^{r-1}}+\frac{1}{2s(s-1)}-1\right)\theta+\varepsilon}
+X^{\frac{\ell-1}{r}+\frac{1}{s}+\varepsilon}
+X^{\frac{\ell}{r}-\frac{\ell}{2^{r}}\left(1-\frac{1}{s}\right)
+\frac{k-1}{k+2}+\theta+\varepsilon}.
\eea
For $\ell/r\ge\frac{3}{2}$, we choose
\bna
\frac{\ell}{r}+\frac{1}{s}-
\left(\frac{\ell}{2^{r-1}}+\frac{1}{2s(s-1)}-1\right)\theta=\frac{\ell}{r}-\frac{\ell}{2^{r}}\left(1-\frac{1}{s}\right)
+\frac{k-1}{k+2}+\theta,
\ena
i.e.,
\bna
\theta=\frac{\big(k+2 \big)\big(s-1 \big)\big( 2^r+\ell(s-1)\big)-2^r\big( k-1\big)\big( s-1\big)s}{\big(k+2 \big)\big(2s(s-1)\ell+2^{r-1} \big)}:=\theta_b.
\ena
Then
\bna
E\ll X^{\frac{\ell}{r}+\frac{1}{s}-\delta_{k, r, s, \ell}+\varepsilon}
\ena
with
\bna
{\delta_{k, r, s, \ell}}=\begin{cases}
\min\left\{\frac{1}{r},\frac{1}{s},\left(\frac{\ell}{2^{r-1}}+\frac{1}{2s(s-1)}-1\right)\theta_b\right\},
& \text{ if } \theta_b\leq \theta_i,  \\
\\
\min\left\{\frac{1}{r},\frac{1}{s},\left(\frac{\ell}{2^{r-1}}+\frac{1}{2(s-1)}-1\right)\theta_i\right\},
& \text{ if } \theta_b > \theta_i.
\end{cases}
\ena
For $\ell/r<\frac{3}{2}$, we have two choices.

\begin{itemize}
\item
 $r=2$, $\ell=2$. In this case, $\frac{\ell}{r}=1$, $\frac{\ell}{2^{r-1}}=1$.
\end{itemize}
 We note that
\bna
\frac{1}{s}-\frac{1}{2}\theta_b\ge\frac{1}{2s(s-1)}\theta_b.
\ena
Then
\bna
E\ll X^{\frac{\ell}{r}+\frac{1}{s}-\delta_{k, r, s, \ell}+\varepsilon}
\ena
with
\bna
{\delta_{k, r, s, \ell}}=\begin{cases}
\frac{1}{2s(s-1)}\theta_{b},
& \text{ if } \theta_b\leq \theta_i,  \\
\\
\frac{1}{2s(s-1)}\theta_{i},
& \text{ if } \theta_b > \theta_i,
\end{cases}
\ena
where
\bna
\theta_{b}=\frac{\big(s-1\big)\big((k+2 )(s+1 )-2( k-1)s\big)}{\big(k+2 \big)\big(2s(s-1)+1 \big)}.
\ena
\begin{itemize}\item
$r=3$, $\ell=4$. In this case, $\frac{\ell}{r}=\frac{4}{3}$, $\frac{\ell}{2^{r-1}}=1$.
\end{itemize}
Note that
\bna
\frac{1}{s}-\frac{1}{6}\theta_b\ge\frac{1}{2s(s-1)}\theta_b.
\ena
Then
\bna
E\ll X^{\frac{\ell}{r}+\frac{1}{s}-\delta_{k, r, s, \ell}+\varepsilon}
\ena
with
\bna
{\delta_{k, r, s, \ell}}=\begin{cases}
\frac{1}{2s(s-1)}\theta_{b},
& \text{ if } \theta_b\leq \theta_i,  \\
\\
\frac{1}{2s(s-1)}\theta_{i},
& \text{ if } \theta_b > \theta_i,
\end{cases}
\ena
where
\bna
\theta_{b}=\frac{\big(s-1\big)\big((k+2 )(s+1 )-2( k-1)s\big)}{\big(k+2 \big)\big(2s(s-1)+1 \big)}.
\ena
(\romannumeral3). $r\ge8$, $2\le s\le 7$.
Recall \eqref{romannumeral3} which we relabel as
\bea\label{romannumeral3*}
\int_\mathfrak{m}T_r^\ell(\alpha,X)T_s(\alpha,X)F(\alpha,X)\mathrm{d}\alpha
\ll X^{\frac{\ell}{r}+\frac{1}{s}-\theta\left(\frac{\ell-2^{r-1}}{2r(r-1)}+\frac{1}{2^{s-1}}\right)+\varepsilon}.
\eea
Note that $2^{r-1}> 2r(r-1)$ for $r\geq 8$. So the upper bound in \eqref{romannumeral3*}
dominates the third term in \eqref{final-estimate1}. By \eqref{zong-fen},
\eqref{zong-fen1}, \eqref{final-estimate1}, \eqref{romannumeral3*} and the implied condition $\ell/r\ge\frac{3}{2}$, we have
\bea\label{err3}
\mathbf{E}=
X^{\frac{\ell}{r}+\varepsilon}
+X^{\frac{\ell}{r}+\frac{1}{s}-\theta\left(\frac{\ell-2^{r-1}}{2r(r-1)}+\frac{1}{2^{s-1}}\right)+\varepsilon}
+X^{\frac{\ell-1}{r}+\frac{1}{s}+\varepsilon}
+X^{\frac{\ell}{r}-\frac{\ell}{2^{r}}\left(1-\frac{1}{s}\right)
+\frac{k-1}{k+2}+\theta+\varepsilon}.
\eea
We choose
\bna
\frac{\ell}{r}+\frac{1}{s}-\theta\left(\frac{\ell-2^{r-1}}{2r(r-1)}+\frac{1}{2^{s-1}}\right)=\frac{\ell}{r}-\frac{\ell}{2^{r}}\left(1-\frac{1}{s}\right)
+\frac{k-1}{k+2}+\theta,
\ena
i.e.,
\bna
\theta=\frac{r\Big(k+2 \Big)\Big(r-1 \Big)\Big( 2^r+\ell\big(s-1\big)\Big)-2^r\Big( k-1\Big)\Big( r-1\Big)rs}{s\Big(k+2 \Big)\Big(2^{r-s+1}\big(r-1\big)r+2^{r-1}\big(\ell-2^{r-1}+2(r-1)r\big) \Big)}:=\theta_c.
\ena
Then
\bna
E\ll X^{\frac{\ell}{r}+\frac{1}{s}-\delta_{k, r, s, \ell}+\varepsilon}
\ena
with
\bna
{\delta_{k, r, s, \ell}}=\begin{cases}
\min\left\{\frac{1}{r},\frac{1}{s},\left(\frac{\ell-2^{r-1}}{2r(r-1)}+\frac{1}{2^{s-1}}\right)\theta_c\right\},
& \text{ if } \theta_c\leq \theta_i,  \\
\\
\min\left\{\frac{1}{r},\frac{1}{s},\left(\frac{\ell-2^{r-1}}{2r(r-1)}+\frac{1}{2^{s-1}}\right)\theta_i\right\},
& \text{ if } \theta_c > \theta_i.
\end{cases}
\ena
(\romannumeral4). $r\ge8$, $s\ge 8$.
Recall \eqref{romannumeral4} which we relabel as
\bea\label{romannumeral4*}
\int_\mathfrak{m}T_r^\ell(\alpha,X)T_s(\alpha,X)F(\alpha,X)\mathrm{d}\alpha
\ll X^{\frac{\ell}{r}+\frac{1}{s}-\theta\left(\frac{\ell-2^{r-1}}{2r(r-1)}+\frac{1}{2s(s-1)}\right)+\varepsilon}.
\eea
Note that $2^{r-1}>2r(r-1)$ for $r\geq 8$. So the upper bound in \eqref{romannumeral4*}
dominates the third term in \eqref{final-estimate1}. By \eqref{zong-fen}
\eqref{zong-fen1}, \eqref{final-estimate1}, \eqref{romannumeral4*} and the implied condition $\ell/r\ge\frac{3}{2}$, we have
\bea\label{err4}
\mathbf{E}=
X^{\frac{\ell}{r}+\varepsilon}
+X^{\frac{\ell}{r}+\frac{1}{s}-\theta\left(\frac{\ell-2^{r-1}}{2r(r-1)}+\frac{1}{2s(s-1)}\right)+\varepsilon}
+X^{\frac{\ell-1}{r}+\frac{1}{s}+\varepsilon}
+X^{\frac{\ell}{r}-\frac{\ell}{2^{r}}\left(1-\frac{1}{s}\right)
+\frac{k-1}{k+2}+\theta+\varepsilon}.
\eea
We choose
\bna
\frac{\ell}{r}+\frac{1}{s}-\theta\left(\frac{\ell-2^{r-1}}{2r(r-1)}+\frac{1}{2s(s-1)}\right)=\frac{\ell}{r}-\frac{\ell}{2^{r}}\left(1-\frac{1}{s}\right)
+\frac{k-1}{k+2}+\theta,
\ena
i.e.,
\bna
\theta=\frac{2r\Big(r-1 \Big)\Big(s-1 \Big)\Big(k+2 \Big)\Big( 2^r+\ell\big(s-1\big)\Big)-2^{r+1}\Big( k-1\Big)\Big( s-1\Big)\Big( r-1\Big)rs}{2^r\Big(k+2 \Big)\Big(\big(r-1\big)r+s\big(s-1\big)\big(\ell-2^{r-1}+2(r-1)r\big) \Big)}:=\theta_d.
\ena
Then
\bna
E\ll X^{\frac{\ell}{r}+\frac{1}{s}-\delta_{k, r, s, \ell}+\varepsilon}
\ena
with
\bna
{\delta_{k, r, s, \ell}}=\begin{cases}
\min\left\{\frac{1}{r},\frac{1}{s},\left(\frac{\ell-2^{r-1}}{2r(r-1)}+\frac{1}{2s(s-1)}\right)\theta_d\right\},
& \text{ if } \theta_d\leq \theta_i,  \\
\\
\min\left\{\frac{1}{r},\frac{1}{s},\left(\frac{\ell-2^{r-1}}{2r(r-1)}+\frac{1}{2s(s-1)}\right)\theta_i\right\},
& \text{ if } \theta_d > \theta_i.
\end{cases}
\ena

\subsection{The case of $\ell \ge 2^{r}$ or $\ell=2^{r}-1$ $(r=s)$}
\qquad\\
By \eqref{final-estimate} and \eqref{fff}, we have
\bea\label{final-estimate2}
&& \int_{\mathfrak{M}} T_{r}^{\ell}(\alpha, X) T_{s}(\alpha, X) F(\alpha, X)  \mathrm{d} \alpha \nonumber\\
&=& \mathbf{M}+
O\Big(X^{\frac{\ell}{r}+\varepsilon}+X^{\frac{\ell}{r}
+\left(\frac{3}{2}-\frac{\ell}{r}\right)\theta+\varepsilon  }
+X^{\frac{\ell}{r}+\frac{1}{s}
-\left(\frac{\ell}{r}+\frac{1}{s}-1\right)\theta+\varepsilon}
+X^{\frac{\ell-1}{r}+\frac{1}{s}+\varepsilon}
+X^{\frac{\ell}{r}-\left(1-\frac{1}{s}\right)+\frac{k-1}{k+2}+\theta+\varepsilon}. \nonumber\\
\eea

Note that the condition $\ell\ge 2^{r}-1$ $(r\ge 2)$ implies $\ell/r\ge\frac{3}{2}$. So the second term is absorbed in the first term in \eqref{final-estimate2}.

(\romannumeral1) $2\le r\le 7$, $2\le s\le 7$.
 By \eqref{zong-fen},
\eqref{zong-fen1}, \eqref{romannumeral1*} and \eqref{final-estimate2}, we have
\bea\label{err1*}
\mathbf{E}=
X^{\frac{\ell}{r}+\varepsilon}
+X^{\frac{\ell}{r}+\frac{1}{s}-
\left(\frac{\ell}{2^{r-1}}+\frac{1}{2^{s-1}}-1\right)\theta+\varepsilon}
+X^{\frac{\ell-1}{r}+\frac{1}{s}+\varepsilon}
+X^{\frac{\ell}{r}-(1-\frac{1}{s})+\frac{k-1}{k+2}+\theta+\varepsilon}.
\eea
We choose
\bna
\frac{\ell}{r}+\frac{1}{s}-\left(\frac{\ell}{2^{r-1}}+\frac{1}{2^{s-1}}-1\right)\theta=\frac{\ell}{r}-(1-\frac{1}{s})+\frac{k-1}{k+2}+\theta,
\ena
i.e,
\bna
\theta=\frac{3\cdot 2^{r-1}}{(k+2)(\ell+2^{r-s})}:=\theta_{e}.
\ena
Then
\bna
E\ll X^{\frac{\ell}{r}+\frac{1}{s}-\delta_{k, r, s, \ell}+\varepsilon}
\ena
with
\bna
{\delta_{k, r, s, \ell}}=\begin{cases}
\min\left\{\frac{1}{r},\frac{1}{s},\left(\frac{\ell}{2^{r-1}}+\frac{1}{2^{s-1}}-1\right)\theta_{e}\right\},
& \text{ if } \theta_e\leq \theta_i,  \\
\\
\min\left\{\frac{1}{r},\frac{1}{s},\left(\frac{\ell}{2^{r-1}}+\frac{1}{2^{s-1}}-1\right)\theta_{i}\right\},
& \text{ if } \theta_e > \theta_i.
\end{cases}
\ena
We take $r=2$, $s=2$.

For $\ell=3$ and $\ell=4$,
\bna
\theta_i=\frac{1}{k+2},\qquad\delta_{k, r, s, \ell}=\frac{(\ell-1)}{2(k+2)}.
\ena
 For $\ell\ge 5$,
\bna
\theta_e=\frac{6}{(k+2)(\ell+1)},\qquad\delta_{k, r, s, \ell}=\frac{3(\ell-1)}{(k+2)(\ell+1)}.
\ena
which recovers the results of Hu and L{\"u} \cite{HL} for $\ell \ge 5$.
The case  of $\ell=3$ and $\ell=4$ is weak for the reason that we specify
$\theta\leq \frac{1}{k+2}$, and then the preceding term is absorbed by the following term in \eqref{O2}. In other words, if we take $\theta=\theta_e$ for $\ell=3$ and $\ell=4$,
\bna
X^{\frac{\ell}{r}-(1-\frac{1}{s})+\frac{k-1}{k+2}+\theta_e}\ll X^{\frac{\ell}{r}-(1-\frac{1}{s})+k\theta_e}.
\ena
(\romannumeral2) $2\le r\le 7$, $s\ge8$.
 By \eqref{zong-fen},
\eqref{zong-fen1}, \eqref{romannumeral2*} and \eqref{final-estimate2}, we have
\bea\label{err2*}
\mathbf{E}=
X^{\frac{\ell}{r}+\varepsilon}
+X^{\frac{\ell}{r}+\frac{1}{s}-
\left(\frac{\ell}{2^{r-1}}+\frac{1}{2s(s-1)}-1\right)\theta+\varepsilon}
+X^{\frac{\ell-1}{r}+\frac{1}{s}+\varepsilon}
+X^{\frac{\ell}{r}-\left(1-\frac{1}{s}\right)+\frac{k-1}{k+2}+\theta+\varepsilon}.
\eea
We choose
\bna
\frac{\ell}{r}+\frac{1}{s}-
\left(\frac{\ell}{2^{r-1}}+\frac{1}{2s(s-1)}-1\right)\theta=\frac{\ell}{r}-\left(1-\frac{1}{s}\right)+\frac{k-1}{k+2}+\theta,
\ena
i.e,
\bna
\theta=\frac{3\cdot 2^{r}\big(s-1\big)s}{\big(k+2\big)\big(2s(s-1)\ell+2^{r-1}\big)}:=\theta_{f}.
\ena
Then
\bna
E\ll X^{\frac{\ell}{r}+\frac{1}{s}-\delta_{k, r, s, \ell}+\varepsilon}
\ena
with
\bna
{\delta_{k, r, s, \ell}}=\begin{cases}
\min\left\{\frac{1}{r},\frac{1}{s},\left(\frac{\ell}{2^{r-1}}+\frac{1}{2s(s-1)}-1\right)\theta_{f}\right\},
& \text{ if } \theta_f\leq \theta_i,  \\
\\
\min\left\{\frac{1}{r},\frac{1}{s},\left(\frac{\ell}{2^{r-1}}+\frac{1}{2s(s-1)}-1\right)\theta_{i}\right\},
& \text{ if } \theta_f > \theta_i.
\end{cases}
\ena
(\romannumeral3) $r\ge8$, $2\le s \le 7$.
 By \eqref{zong-fen},
\eqref{zong-fen1}, \eqref{romannumeral3*} and \eqref{final-estimate2}, we have
\bea\label{err3*}
\mathbf{E}=
X^{\frac{\ell}{r}+\varepsilon}
+X^{\frac{\ell}{r}+\frac{1}{s}-
\left(\frac{\ell-2^{r-1}}{2r(r-1)}+\frac{1}{2^{s-1}}\right)\theta+\varepsilon}
+X^{\frac{\ell-1}{r}+\frac{1}{s}+\varepsilon}
+X^{\frac{\ell}{r}-\left(1-\frac{1}{s}\right)+\frac{k-1}{k+2}+\theta+\varepsilon}.
\eea
We choose
\bna
\frac{\ell}{r}+\frac{1}{s}-
\left(\frac{\ell-2^{r-1}}{2r(r-1)}+\frac{1}{2^{s-1}}\right)\theta=\frac{\ell}{r}-\left(1-\frac{1}{s}\right)+\frac{k-1}{k+2}+\theta,
\ena
i.e,
\bna
\theta=\frac{3\cdot 2^{s}\Big(r-1\Big)r}{\Big(k+2\Big)\Big(2\big(r-1\big)r+2^{s-1}\big(\ell-2^{r-1}+2(r-1)r\big)\Big)}:=\theta_{g}.
\ena
Then
\bna
E\ll X^{\frac{\ell}{r}+\frac{1}{s}-\delta_{k, r, s, \ell}+\varepsilon}
\ena
with
\bna
{\delta_{k, r, s, \ell}}=\begin{cases}
\min\left\{\frac{1}{r},\frac{1}{s},\left(\frac{\ell-2^{r-1}}{2r(r-1)}+\frac{1}{2^{s-1}}\right)\theta_g\right\},
& \text{ if } \theta_g\leq \theta_i,  \\
\\
\min\left\{\frac{1}{r},\frac{1}{s},\left(\frac{\ell-2^{r-1}}{2r(r-1)}+\frac{1}{2^{s-1}}\right)\theta_i\right\},
& \text{ if } \theta_g > \theta_i.
\end{cases}
\ena
(\romannumeral4) $r\ge8$, $s\ge8$.
 By \eqref{zong-fen},
\eqref{zong-fen1}, \eqref{romannumeral4*} and \eqref{final-estimate2}, we have
\bea\label{err4*}
\mathbf{E}=
X^{\frac{\ell}{r}+\varepsilon}
+X^{\frac{\ell}{r}+\frac{1}{s}-
\left(\frac{\ell-2^{r-1}}{2r(r-1)}+\frac{1}{2s(s-1)}\right)\theta+\varepsilon}
+X^{\frac{\ell-1}{r}+\frac{1}{s}+\varepsilon}
+X^{\frac{\ell}{r}-\left(1-\frac{1}{s}\right)+\frac{k-1}{k+2}+\theta+\varepsilon}.
\eea
We choose
\bna
\frac{\ell}{r}+\frac{1}{s}-
\left(\frac{\ell-2^{r-1}}{2r(r-1)}+\frac{1}{2s(s-1)}\right)\theta=\frac{\ell}{r}-\left(1-\frac{1}{s}\right)+\frac{k-1}{k+2}+\theta,
\ena
i.e,
\bna
\theta=\frac{6\Big(r-1\Big)\Big(s-1\Big)rs}{\Big(k+2\Big)\Big(\big(r-1\big)r+s\big(s-1\big)\big(\ell-2^{r-1}+2(r-1)r\big)\Big)}:=\theta_{h}.
\ena
Then
\bna
E\ll X^{\frac{\ell}{r}+\frac{1}{s}-\delta_{k, r, s, \ell}+\varepsilon}
\ena
with
\bna
{\delta_{k, r, s, \ell}}=\begin{cases}
\min\left\{\frac{1}{r},\frac{1}{s},\left(\frac{\ell-2^{r-1}}{2r(r-1)}+\frac{1}{2s(s-1)}\right)\theta_h\right\},
& \text{ if } \theta_h\leq \theta_i,  \\
\\
\min\left\{\frac{1}{r},\frac{1}{s},\left(\frac{\ell-2^{r-1}}{2r(r-1)}+\frac{1}{2s(s-1)}\right)\theta_i\right\},
& \text{ if } \theta_h > \theta_i.
\end{cases}
\ena
We finish the proof of Proposition \ref{pro}.

\bigskip

{\small \textsc{Qingfeng Sun},
	\textsc{School of Mathematics and Statistics, Shandong University, Weihai,
		Weihai, Shandong 264209, China}\\
	\indent{\it E-mail address}: qfsun@sdu.edu.cn

\medskip

{\small \textsc{Chenhao Du},
	\textsc{School of Mathematics and Statistics, Shandong University, Weihai,
		Weihai, Shandong 264209, China}\\
	\indent{\it E-mail address}: chenhaodu@mail.sdu.edu.cn
	
\end{document}